\documentclass[10pt]{amsart}
\usepackage[letterpaper,hmargin=1in,vmargin=1.25in]{geometry}
\usepackage{amsmath,amssymb,amsthm,amsfonts}
\usepackage{graphicx}
\usepackage{color}

\newtheorem{theorem}{Theorem}[section]
\newtheorem{lemma}[theorem]{Lemma}
\theoremstyle{definition}
\newtheorem{algorithm}[theorem]{Algorithm}

\newcommand{\gp}[1]{{\left\langle #1 \right\rangle}}

\def\MF{{\mathbb{F}}}
\def\MZ{{\mathbb{Z}}}
\def\MN{{\mathbb{N}}}

\def\ovt{{\bf t}}

\title{Cryptanalysis of Anshel-Anshel-Goldfeld-Lemieux key agreement protocol}

\date{Version 4 $\bullet$ September 30, 2007}

\author[]{Alex D. Myasnikov}
\address{Department of Mathematics, Stevens Institute of Technology,
Hoboken, NJ 07030} \email{amyasnikov@stevens.edu}

\author[]{Alexander Ushakov}
\address{Department of Mathematics, Stevens Institute of Technology,
Hoboken, NJ 07030} \email{aushakov@stevens.com}

\begin{document}

\maketitle

\begin{abstract}
The Anshel-Anshel-Goldfeld-Lemieux (abbreviated AAGL)
key agreement protocol \cite{AAG2} is proposed to be used
on low-cost platforms which constraint the use of computational
resources.
The core of the protocol is the concept of an Algebraic Eraser$^{TM}$
(abbreviated AE) which is claimed to be a suitable primitive for use within
lightweight cryptography.
The AE primitive is based on a new and ingenious idea of using an action of a semidirect
product on a (semi)group to obscure involved algebraic structures.
The underlying motivation for AAGL protocol
is the need to secure networks
which deploy Radio Frequency Identification (RFID) tags used for identification,
authentication, tracing and point-of-sale applications.

In this paper we revisit the computational problem on which AE relies
and heuristically analyze its hardness.
We show that for proposed
parameter values it is impossible to instantiate the secure protocol.
To be more precise, in $100\%$ of randomly generated instances of the protocol
we were able to find a secret conjugator $z$ generated by TTP algorithm
(part of AAGL protocol).
\end{abstract}

\section{The Colored Burau Key Agreement Protocol}

A general mathematical framework of AAGL protocol
is quite complicated. In this paper we try to omit unnecessary details and
simplify the notation of \cite{AAG2} as much as possible.
We refer an interested reader to \cite[Sections 2 and 3]{AAG2}
for a complete description. Here we start out
by giving a particular implementation of the primitive
called the Colored Burau Key Agreement Protocol (CBKAP).

\subsection{A platform group}
\label{se:platform_group}

Fix an integer $n \ge 7$ and a prime $p$.
Let $\ovt = (t_1, \ldots ,t_n)$ be a tuple of formal variables.
Define matrices
$$
x_1(\ovt) =
\left(
\begin{array}{lllll}
-t_1 & 1 & & \\
     & 1 & & \\
     &   & \ddots & \\
     &   & & 1 \\
\end{array}
\right)
$$
and for $i=2,\ldots,n-1$
$$
x_i(\ovt) =
\left(
\begin{array}{lllll}
1 & & & \\
  & \ddots  & & \\
  &  t_i & -t_{i} & 1 \\
  & &   & \ddots & \\
  & &   & & 1 \\
\end{array}
\right)
$$
which is the identity matrix except for the $i$th row where it has successive entries
$t_i$, $-t_i$, $1$ with $-t_i$ on the diagonal.
We look at the matrices
    $x_1(\ovt), \ldots, x_{n-1}(\ovt)$
as elements of the group $GL(n,\MF_p(\ovt))$ of $n \times n$ matrices
with entries as Laurent polynomials over the finite field $\MF_p$.
The symmetric group on $n$ symbols $S_n$ acts on $GL(n,\MF_p(\ovt))$
by permuting the variables $t_1, \ldots t_n$. We denote the result
of the action of $s\in S_n$ on $x \in GL(n,\MF_p(\ovt))$ by
$^sx$.

The semidirect product $GL(n,\MF_p(\ovt)) \rtimes S_n$
of the groups $GL(n,\MF_p(\ovt))$ and $S_n$ relative
to the defined action of $S_n$ on matrices $GL(n,\MF_p(\ovt))$ is
a set of pairs
    $$\{ (m,s) \mid m\in GL(n,\MF_p(\ovt)) ,~ s \in S_n \}$$
with multiplication given by
    $$(m_1,s_1) \cdot (m_2,s_2) := (m_1 \cdot ^{s_1}m_2, s_1 \cdot s_2).$$
Denote by $s_i = (i,i + 1) \in S_n$ the transposition which
interchanges $i$ and $i+1$ and by $g_i$ the element of
the semidirect product $GL(n,\MF_p(\ovt)) \rtimes S_n$
    $$g_i = (x_i(\ovt),s_i).$$
A subgroup
    $$G = \gp{ g_1, \ldots, g_{n-1} }$$
of $GL(n,\MF_p(\ovt)) \rtimes S_n$ is called the {\em colored Burau
group}. The group $G$ is a platform group for AAGL key agrement protocol.

Recall that the group $B_n$ of $n$-strand braids has the classical Artin's presentation:
\begin{displaymath}
B_n =
 \left\langle
\begin{array}{lcl}\sigma_1,\ldots,\sigma_{n-1} & \bigg{|} &
\begin{array}{ll}
 \sigma_i \sigma_j \sigma_i= \sigma_j \sigma_i \sigma_j & \textrm{if }|i-j|=1 \\
 \sigma_i \sigma_j=\sigma_j \sigma_i & \textrm{if }|i-j|>1
\end{array}
\end{array}
\right\rangle.
\end{displaymath}
A word over the group alphabet $\{\sigma_1,\ldots,\sigma_{n-1}\}$ is called a
{\em braid word}.
Any $n$-strand braid can be represented by a braid word.
The length of a shortest braid word representing an element $g \in B_n$
is called the {\em geodesic length} of $g$ relative to the Artin's set of generators
and is denoted by $|g|$. The function $|\cdot| : B_n \rightarrow \MN$ is called
the {\em geodesic length function} on $B_n$.

\begin{lemma} \label{le:braid_rels}
The elements $g_i = (x_i(\ovt), s_i)$,
for $i = 1, 2, \ldots , n-1$, satisfy the braid
relations and hence determine a representation of the braid group $B_n$,
i.e., the mapping $\sigma_i \stackrel{\varphi}{\mapsto} g_i$ defines
a group epimorphism
    $$\varphi: B_n \rightarrow G.$$
\end{lemma}

\begin{proof}
Straightforward check.
\end{proof}

\subsection{Action of the platform group on $GL(n,\MF_p)$}

Fix elements $\tau_1, \ldots, \tau_n \in \MF_p$ and define a
homomorphism $\pi$ which maps $GL(n,\MF_p(\ovt))$ into $GL(n,\MF_p)$
by assigning the value $\tau_i$ to the variable $t_i$, i.e., by
evaluating a matrix at $\tau_1, \ldots, \tau_n$. We call $\pi$ the
{\em evaluation function}.
\begin{quote}
{\bf Assumption on $\tau_1, \ldots, \tau_n$.} We assume that $\pi$
defines a correct group homomorphism.
\end{quote}
Relative to the chosen tuple $\tau_1, \ldots, \tau_n \in \MF_p$
and the corresponding function $\pi$ one can define an action of
$GL(n,\MF_p(\ovt)) \rtimes S_n$ on $GL(n,\MF_p) \times S_n$ by putting
    $$(m_1,s_1) \star (m_2,s_2) = (m_1 \cdot \pi(^{s_1}m_2) , s_1 s_2)$$
where $\star$ denotes the action. Indeed, it is not difficult to check that
$\star$ is an action and satisfies the property
    $$((m_1,s_1) \star (m_2,s_2) ) \star (m_3,s_3) = (m_1,s_1) \star ( (m_2,s_2) \cdot (m_3,s_3) ).$$
We say that $(m_1,t_1)$ and $(m_2,t_2)$ {\em $\star$-commute} if the equality
    $$(\pi(m_1),s_1) \star (m_2,s_2) = (\pi(m_2),s_2) \star (m_1,s_1)$$
holds. The next lemma is obvious.

\begin{lemma}
Let $w = \prod_{k=1}^m (x_{i_k}(\ovt),s_{i_k})$
and $v = \prod_{p=1}^l (x_{j_p}(\ovt),s_{j_p})$
be such that $|i_k-j_p|>1$ for every $k=1,\ldots,m$ and
$p=1,\ldots,l$. Then the elements $w$ and $v$ $\star$-commute.
\end{lemma}

\subsection{The protocol}

Before the parties perform actual transmissions the following data is being prepared
by the Third Trusted Party (TTP).
\begin{itemize}
    \item
A matrix $m_0 \in GL(n,\MF_p)$ which has an irreducible
characteristic polynomial over $\MF_p$. The choice of $m_0$ is not relevant
for the purposes of this paper, we refer the reader to \cite{AAG2}
for more information on how $m_0$ can be generated randomly.
    \item
$\star$-commuting subgroups $A = \gp{w_1, \ldots, w_\gamma}$ and $B = \gp{u_1, \ldots, u_\gamma}$ of
the group $G$.
We want to point out that the elements $w_i$ and $v_j$
are given to us as products of generators of $G$ and there inverses, i.e., as formal
words in group alphabet $\{g_1,\ldots, g_{n-1}\}$. We prefer this form because it allows us to
avoid time consuming matrix multiplication in $GL(n,\MF_p(\ovt))$.
\end{itemize}
Both, the matrix $m_0$ and subgroups $A$ and $B$, can be chosen only once.
Now, the public and private keys are chosen as follows:

\noindent{\bf Alice's Private Key:} is a pair which consists of a matrix of the form
    $$n_a = l_1 m_0^{\alpha_1} + l_2 m_0^{\alpha_2} + \ldots + l_r m_0^{\alpha_r} \in GL(n,\MF_p)$$
(where $l_1 , \ldots, l_r \in \MF_p$ and
$r, \alpha_1, \ldots, \alpha_r \in \MZ^+$) and
a random sequence $w_{i_1}^{\varepsilon_1}, \ldots, w_{i_m}^{\varepsilon_m}$ of generators of $A$
and their inverses.

\noindent{\bf Alice's Public Key:}
is an element
    $$A_{public} = (n_a,id) \star w_{i_1}^{\varepsilon_1} \star \ldots \star w_{i_m}^{\varepsilon_m} \in GL(n,\MF_p) \times S_n.$$
Recall that each $w_{i_k}$ is given as a formal product of the generators of $G$.
To perform the $\star$-operation efficiently one should not directly compute $w_{i_k}$,
but consequently apply the factors of $w_{i_k}$ to the argument.

\noindent{\bf Bob's Private Key:} is a pair which consists of a matrix of the form
    $$n_b = l_1' m_0^{\beta_1} + l_2' m_0^{\beta_2} + \ldots + l_{r'}' m_0^{\beta_{r'}} \in GL(n,\MF_p)$$
(where $l_1' , \ldots, l_{r'}' \in \MF_p$ and
$r', \beta_1, \ldots, \beta_{r'} \in \MZ^+$) and
a random sequence $v_{j_1}^{\delta_1}, \ldots, v_{j_l}^{\delta_l}$ of generators of $B$
and their inverses.

\noindent{\bf Bob's Public Key:}
is a pair
    $$B_{public} = (n_b,id) \star v_{j_1}^{\delta_i} \star \ldots \star v_{j_l}^{\delta_l} \in GL(n,\MF_p) \times S_n.$$
Again, each $v_{j_k}$ is given as a formal product of the generators of $G$.
To perform the $\star$-operation efficiently one should not directly compute $v_{j_k}$,
but consequently apply the factors of $v_{j_k}$ to the argument.

\noindent{\bf The shared key:}
is an element of $GL(n,\MF_p) \times S_n$ obtained by Alice in the form
    $$[(n_a,id) \cdot B_{public}] \star w_{i_1}^{\varepsilon_1} \star \ldots \star w_{i_m}^{\varepsilon_m}$$
and by Bob in the form
    $$[(n_b,id) \cdot A_{public}] \star v_{j_1}^{\delta_i} \star \ldots \star v_{j_l}^{\delta_l}$$
It requires a little work to prove that
the obtained elements are indeed equal in $GL(n,\MF_p)$.
We omit the proof.

\subsection{TTP algorithm}

The cornerstone part of the proposed key-exchange is the choice of
$\star$-commuting subgroups of the group $G$. The basic idea is to use Lemma
\ref{le:braid_rels} and choose commuting subgroups $A$ and $B$ in $B_n$
and then pull them into $G$ using the epimorphism $\varphi$. The
resulting subgroups $\varphi(A)$ and $\varphi(B)$ of $G$ commute.
Moreover, for any choice of $\pi$ the subgroups
$\varphi(A)$ and $\varphi(B)$ $\star$-commute.

Before we present the algorithm we need to give some details about
the braid group $B_n$. The group $B_n$ has a cyclic center generated by
an element $\Delta^2$ where $\Delta$ is an element called the half twist
and can be expressed in the generators of $B_n$ as follows:
    $$\Delta = (\sigma_1 \ldots \sigma_{n-1}) \cdot (\sigma_1 \ldots \sigma_{n-2}) \cdot \ldots \cdot (\sigma_1).$$
Any element $g \in B_n$ can be uniquely represented in a form
    $$\Delta^p \xi_1 \ldots \xi_p$$
satisfying certain conditions and called the left {\em Garside
normal form}.

Now, since $\Delta^2$ is a central element it follows that
element $u,w$ commute in $B_n$ if and only $u\Delta^{2p}$
and $w\Delta^{2r}$ do (for any choice of $p,r \in \MZ$). Hence, we may always assume that
the normal forms of the generators $\{w_1,\ldots, w_\gamma\}$
and $\{v_1,\ldots, v_\gamma\}$ have the power of $\Delta$
equal to $0$ or $-1$. When we say that we reduce a braid
modulo $\Delta^2$ we mean changing the $\Delta$-power of its normal form to
$-1$ or $0$ depending on parity.

The algorithm below (originally proposed in \cite{AAG2})
generates two $\star$-commuting subgroups.

\begin{algorithm}\label{alg:ttp}{\bf (TTP algorithm)}
\begin{itemize}
    \item[(1)]
Choose two secret subsets $BL = \{b_{l_1}, \ldots, b_{l_\alpha}\}$,
$BR = \{b_{r_1}, \ldots, b_{r_\beta}\}$ of the set of
generators of $B_n$, where $|l_i-r_j|\ge 2$ for all $1\le i\le l_\alpha$
and $1\le j \le r_\beta$.
    \item[(2)]
Choose a secret element $z \in B_n$.
    \item[(3)]
Choose words $\{w_1, \ldots, w_\gamma\}$ of bounded length over the generators $BL$.
    \item[(4)]
Choose words $\{v_1, \ldots, v_\gamma\}$ of bounded length over the generators $BR$.
    \item[(5)]
For each $i = 1, \ldots, \gamma$:
\begin{itemize}
    \item[(a)]
calculate the left normal form of $zw_iz^{-1}$ and reduce the result modulo $\Delta^2$;
    \item[(b)]
put $w_i'$ to be a braid word corresponding to the
element calculated in (a);
    \item[(c)]
calculate the left normal form of $zv_iz^{-1}$ and reduce the result modulo $\Delta^2$;
    \item[(d)]
put $v_i'$ to be a braid word corresponding to the
element calculated in (c).
\end{itemize}
    \item[(6)]
Publish the sets $\{v_1', \ldots, v_\gamma'\}$ and $\{w_1', \ldots, w_\gamma'\}$.
\end{itemize}
\end{algorithm}

We want to point out that TTP algorithm produces generators of two commuting subgroups in $B_n$.
Alice and Bob need to compute their images in $GL(n,\MF_p(\ovt))$ to obtain $\star$-commuting subgroups.

\subsection{Security assumptions}
\label{sec:sec_assum}

It was noticed in \cite{AAG2} that if the conjugator $z$ generated randomly by
TTP algorithm is known then there exists an efficient linear attack on the scheme
which is able to recover the shared key of the parties.
The problem of recovering the exact $z$ seems like a very difficult mathematical problem because
it reduces to solving the system of equations
\begin{equation}\label{eq:system}
\left\{
\begin{array}{l}
w_1' = \Delta^{2p_1} z w_1 z^{-1} \\
\ldots \\
w_\gamma' = \Delta^{2p_\gamma} z w_\gamma z^{-1} \\
v_1' = \Delta^{2r_1} z v_1 z^{-1} \\
\ldots \\
v_\gamma' = \Delta^{2r_\gamma} z v_\gamma z^{-1} \\
\end{array}
\right.
\end{equation}
which has too many unknowns, only left hand sides (i.e., elements $w_1', \ldots, w_\gamma'$, $v_1', \ldots, v_\gamma'$)
are known. Hence, it might be difficult to find the original $z$.


Now observe that the AAGL key exchange protocol uses only the output of TTP algorithm,
namely the tuples $\{v_1', \ldots, v_\gamma'\}$ and
$\{w_1', \ldots, w_\gamma'\}$ since all internal
values in TTP algorithm are not available to the parties.
In other words
it is irrelevant for the protocol how two particular commuting generating sets
were constructed.
This observation leads us to the following problem
\begin{quote}
For tuples $\{v_1', \ldots, v_\gamma'\}$ and $\{w_1', \ldots, w_\gamma'\}$
find any $z'$ and any numbers $p_1,\ldots, p_\gamma, r_1,\ldots, r_\gamma \in \MZ$
such that the words $\{\Delta^{2p_1} z'^{-1} v_1' z',\ldots, \Delta^{2p_\gamma}z'^{-1}v_\gamma' z'\}$
and $\{\Delta^{2r_1} z'^{-1}w_1' z',\ldots, \Delta^{2r_\gamma}z'^{-1} w_\gamma' z'\}$
can be expressed as words over two disjoint commuting subsets of generators of $B_n$.
\end{quote}
This is a new problem for computational group theory. Let us refer to it as
{\em simultaneous conjugacy separation search problem} (abbreviated SCSSP).
We want to emphasize that SCSSP has little in common with the {\em simultaneous
conjugacy search problem} often referenced in the papers on the braid group cryptanalysis.
The main difference is that in the conjugacy search problem both conjugate elements are available
and the goal is to recover the secret conjugator. And in case of SCSSP only the
left side of the equation is known. It is not clear if one of the problems can be reduced
to the other.

It follows from the observation above that {\em any solution $z'$ to a problem stated above plays a role
of a conjugator $z$ and can be used in a linear attack outlined in \cite{AAG2}.}
The main goal of this paper is to present an algorithm which for proposed parameter values
solves SCSSP. Experimental results convince us that our attack is a serious threat for AAGL
as the success rate is $100\%$. Furthermore, a slight modification of the algorithm produces
the exact $z$ generated by TTP in $40\%$ of randomly generated instances.

\subsection{Proposed parameter values}
\label{sec:param_values}
To provide $80$ bits of security against the exhaustive search for $z$
for the scheme the authors propose two slightly different sets of parameters:
\begin{itemize}
    \item
{\bf Parameter set \# 1.}
\begin{itemize}
    \item
Let $n = 14$, $p=13$, and $r=3$.
    \item
Choose the conjugator $z$ randomly of length $17$.
    \item
Choose the words $w_i$ and $v_j$ randomly of length approximately $10$.
    \item
The number $\gamma$ of the words $w_i$ and $v_j$ is $27$.
\end{itemize}
    \item
{\bf Parameter set \# 2.}
\begin{itemize}
    \item
Let $n = 12$, $p=13$, and $r=3$.
    \item
Choose the conjugator $z$ randomly of length $18$.
    \item
Choose the words $w_i$ and $v_j$ randomly of length approximately $10$.
    \item
The number $\gamma$ of the words $w_i$ and $v_j$ is $27$.
\end{itemize}
\end{itemize}

\section{TTP attack}

In this section we describe a heuristic  attack which finds a solution to a given instance of SCSSP.
The main ingredient in our attack is a length function on the group $B_n$.
As it is explained in \cite{MU} there are no known efficiently computable and "sharp"
length functions for braid groups.
Therefore, for our attack we adopt the method of approximation of the geodesic length function
originally proposed in \cite{MSU1}.
In all our algorithms by $|\cdot|$ we denote approximation of the geodesic length function.

 We present results of experiments which show that a fast heuristic procedure
 based on the length-based reduction is extremely successful for the suggested
 parameters.
 In fact, {\em every} instance of TTP algorithm generated in our experiments has been broken.

\subsection{Generation}\label{sec:param_gens}

The original paper \cite{AAG2} lacks any
details on how to randomly generate the secret element $z$ and the
words $\{w_1, \ldots, w_{\gamma}\}$, $\{v_1, \ldots, v_{\gamma}\}$ in TTP algorithm.
Hence, in all our experiments:
\begin{itemize}
    \item
The word $z$ is taken uniformly randomly as a word of a particular length from the
ambient free group $F(\sigma_1, \ldots, \sigma_{n-1})$.
    \item
The words $w_1, \ldots, w_{\gamma}$ and $v_1, \ldots, v_{\gamma}$
are taken uniformly randomly as words of particular lengths from the
ambient free groups $F(BL)$ and $F(BR)$.
\end{itemize}

Also, the authors suggest to take the sets $BL$ and $BR$ randomly on step
(1) of TTP algorithm. Observe that in general this might result in a choice of
$BL$ such that for some $1\le i < j < k\le n-1$
    $$\sigma_i,\sigma_k \in BL, \mbox{ but } \sigma_j \in BR.$$
We think that this situation is not desirable as it excludes the use of at least two
braid generators in the words $w_i$ and $v_j$. We think that
the choice of the following sets
    $$BL = \{\sigma_1, \ldots, \sigma_{l}\} \mathrm{\ and\ } BR =\{\sigma_{l+2}, \ldots, \sigma_{n-1}\}$$
(where $n$ is an even number and $l=(n-2)/2$) is optimal
as it excludes only $\sigma_{l+1}$ which maximizes the size of a space for the words $w_1, \ldots, w_{\gamma}$ and $v_1, \ldots, v_{\gamma}$.

\subsection{Recovering $\Delta$-powers}

The first stage in our attack is recovering $\Delta$ powers in the system
(\ref{eq:system}), i.e., computing numbers $p_1, \ldots, p_\gamma$ and $r_1, \ldots, r_\gamma$.
The main tool in our computations below is the triangular inequality for the Cayley graph of the braid group $B_n$.
Observe that the following inequalities hold.
\begin{itemize}
    \item[]
{\bf (Parameter set \#1)}
For each $i = 1,\ldots, \gamma$
    $$|z^{-1} u_i z|\le 2|z|+|u_i| = 44 ~~\mbox{ and }~~ |z^{-1} w_j z| \le 2|z|+|w_j| = 44$$
and
    $$|\Delta^{2p}| = pn(n-1) = 182p.$$
Hence,
$|\Delta^{2p} z^{-1} u_i z|, |\Delta^{2p} z^{-1} w_j z| \in [182p-44,182p+44]$
and
    $$|\Delta^{2p} z^{-1} u_i z| - |\Delta^{2(p-1)} z^{-1} u_i z| \ge 182-2\cdot 44 = 94$$
    $$|\Delta^{2p} z^{-1} w_j z| - |\Delta^{2(p-1)} z^{-1} w_j z| \ge 182-2\cdot 44 = 94.$$
    \item[]
{\bf (Parameter set \#2)}
For each $i = 1,\ldots, \gamma$
    $$|z^{-1} u_i z|\le 2|z|+|u_i| = 46 ~~\mbox{ and }~~ |z^{-1} w_j z| \le 2|z|+|w_j| = 46$$
and
    $$|\Delta^{2p}| = pn(n-1) = 132p.$$
Hence $|\Delta^{2p} z^{-1} u_i z|, |\Delta^{2p} z^{-1} w_j z|  \in [132p-46,132p+46]$ and
    $$|\Delta^{2p} z^{-1} u_i z| - |\Delta^{2(p-1)} z^{-1} u_i z| \ge 132-2\cdot 46 = 40$$
    $$|\Delta^{2p} z^{-1} w_j z| - |\Delta^{2(p-1)} z^{-1} w_j z| \ge 132-2\cdot 46 = 40.$$
\end{itemize}
This observation implies that for both parameter sets the sequences
$\{|\Delta^{2p} z^{-1} u_i z|\}_{p=0}^\infty$ and
$\{|\Delta^{2p} z^{-1} w_j z|\}_{p=0}^\infty$ are
strictly increasing.
Thus, to recover the original power of $\Delta$
one can repeatedly multiply $u_i'$ (and $w_j'$) on the left by $\Delta^2$
until the length cannot be reduced anymore (see Algorithm \ref{alg:d-red}).
Moreover, since the difference between two elements differing by $\Delta^2$
is at least $40$ even crude approximations of the length function must work.

\begin{algorithm}[$\Delta$-power recovery]\label{alg:d-red} \ \\
{\sc Input:} An element $w \in B_n$.\\
{\sc Output:} An element $u$ minimal in the left coset $\gp{\Delta^{-2}} w$.\\
{\sc Computations:}
\begin{enumerate}
    \item[A.]
Set $u = w$.
    \item[B.]
If $|u| > |\Delta^{-2}u|$ then set $u = \Delta^{-2}u$ and goto B.
    \item[C.]
If $|u| > |\Delta^{2}u|$ then set $u = \Delta^{2}u$ and goto B.
    \item[D.]
Otherwise output $u$.
\end{enumerate}
\end{algorithm}

Clearly Algorithm \ref{alg:d-red} always terminates. The time complexity of the algorithm depends on
the complexity of the procedure which approximates the geodesic length. The procedure is heuristic and
its worst case complexity is not known. Experimental results in \cite{MSU1} suggest that the length approximation can be
efficiently computed for most braid words and we estimate the expected complexity of the procedure as $O(n)$.
Under this assumption, it is easy to see that the power-recovery algorithm can be executed in at most
$O((|w|+n^2)|w|/n^2) = O(|w|^2/n^2+|w|)$ steps as the algorithm performs up to $|w|/n^2$ iterations and on each iteration
for a word $u$ of length up to $|w|$ the length of a word $\Delta^2u$ is estimated.

\subsection{Recovering conjugator}

The second part of the attack computes a secret conjugator.
At this point we assume that all $\Delta$-powers from the system
(\ref{eq:system}) are successfully found and we have a system of equations of the form
\begin{equation}\label{eq:system2}
\left\{
\begin{array}{l}
w_1'' = z w_1 z^{-1} \\
\ldots \\
w_\gamma'' = z w_\gamma z^{-1} \\
v_1'' = z v_1 z^{-1} \\
\ldots \\
v_\gamma'' = z v_\gamma z^{-1} \\
\end{array}
\right.
~~\mbox{ or }~~
\left\{
\begin{array}{l}
z^{-1} w_1'' z = w_1 \\
\ldots \\
z^{-1} w_\gamma'' z = w_\gamma \\
z^{-1} v_1'' z = v_1 \\
\ldots \\
z^{-1} v_\gamma'' z = v_\gamma \\
\end{array}
\right.
\end{equation}
where only elements $u_i'' = \Delta^{-2p_i} u_i'$ and $w_j'' = \Delta^{-2r_i} w_j'$ are known.
Let us call two sets of braids {\em separated} if
they can be expressed as words over disjoint commuting  sets of generators of $B_n$.
As mentioned in Section \ref{sec:sec_assum} to break the protocol it is
sufficient to find any conjugator $z'$ which conjugates two tuples of
elements $(u_1'',\ldots,u_\gamma'')$ and $(w_1'', \ldots,w_\gamma'')$
into two separated tuples of elements $(u_1,\ldots,u_\gamma)$ and $(w_1, \ldots,w_\gamma)$.
This is the main goal of our attack.

Let $\bar{u} = (u_1,\ldots,u_m)$ be a tuple of elements in $B_n$ and
$x$ is an element of $B_n$. Denote by $|\bar{u}|$ the total length
of elements in $\bar{u}$, i.e., put
    $$|\bar{u}| = \sum_{i=1}^m |u_i|.$$
Denote by $\bar{u}^x$ a tuple obtained from $\bar{u}$ by conjugation of each
its element by $x$.
It is intuitively clear that conjugation of a tuple of braids
by a random element $x$ almost always increases the length of
the tuple. In other words, for a random element $x$ the inequality
\begin{equation}\label{eq:assume_length}
|\bar{u}^x| > |\bar{u}|
\end{equation}
is {\em almost always} true.
We do not have a proof of this fact, but numerous
experiments convince us that it is true. Moreover, conjugation by
longer elements almost always results in longer tuples.

The idea that conjugation consequently increases the length of tuples
is not new. It was used in papers \cite{HT}, \cite{GKTTV} for different length functions
with different success. But the most successful is a recent attack \cite{MU} which
uses approximation of the geodesic length.
In this paper we use the idea of separating two tuples of braids.
To find $z'$ we repeatedly conjugate the tuple
    $(u_1'',\ldots,u_\gamma'',w_1'', \ldots,w_\gamma'')$
by generators of $B_n$ and their inverses
and if for some generator $\sigma_k^{\pm 1}$ the decrease of the total length of the tuple
is observed then it is reasonable to guess that $\sigma_k^{\pm 1}$ is involved in $z'$.

\begin{algorithm}[Recovering conjugator - I]
\label{alg:z_red}\ \\
{\sc Input:} Tuples $\bar{a} = \{a_1, \ldots, a_{\gamma}\}$ and $\bar{b} = \{b_1, \ldots, b_{\gamma}\}$.\\
{\sc Output:} An element $z'$ separating tuples $\bar{a}$ and $\bar{b}$.\\
{\sc Initialization:} Set $z' = 1$. \\
{\sc Computations:}
\begin{enumerate}
    \item[A.]
For each $i = 1,\ldots, n-1$ and $\varepsilon = \pm 1$
conjugate tuples $\bar{a}$ and $\bar{b}$ by a generator
$\sigma_i^\varepsilon$ and compute
    $$\delta_{i,\varepsilon} = |\bar{a}^{\sigma_i^\varepsilon}| + |\bar{b}^{\sigma_i^\varepsilon}| - ( |\bar{a}| + |\bar{b}| ).$$
    \item[B.]
If for some $\sigma_i^\varepsilon$ the sets
$\bar{a}^{\sigma_i^\varepsilon}$ and $\bar{b}^{\sigma_i^\varepsilon}$ are separated
then output $z' = \sigma_i^\varepsilon z'$.
    \item[C.]
Otherwise, if all $\delta_{i,\varepsilon}$ are positive (i.e., conjugation by $\sigma_i^\varepsilon$ cannot further decrease the total length)
then output FAILURE.
    \item[D.]
Otherwise choose $i$ and $\varepsilon$ for which $\delta_{i,\varepsilon}$ is minimal.
Set $z' = \sigma_i^\varepsilon z'$, $\bar{a} = \bar{a}^{\sigma_i^\varepsilon}$, and
$\bar{b} = \bar{b}^{\sigma_i^\varepsilon}$. Goto step A.
\end{enumerate}
\end{algorithm}

The described attack is similar to the one described in \cite{MU}.
Recall that the main problem in \cite{MU} was the existence of so-called {\em peaks}
(see \cite[Definition 2.5]{MU}).
This phenomenon is a consequence of difficult structure of finitely generated
subgroups of braid groups. In this paper, we do not have this problem as $z$
is chosen in the whole group $B_n$.

Note that Algorithm \ref{alg:z_red} is a greedy descend procedure. It may fail due to
the fact that there exists a small fraction of words for which
the inequality (\ref{eq:assume_length}) does not hold. It is also prone
to the length approximation errors. One can significantly reduce
the failure rate of a descent procedure by
introducing a backtracking algorithm which allows exploration of more than one search paths.
Algorithm \ref{alg:z_red_BACK} gives an implementation of the attack with backtracking.

\begin{algorithm}[Recovering conjugator with Backtracking]
\label{alg:z_red_BACK}\ \\
{\sc Input:} Tuples $\bar{a} = \{a_1, \ldots, a_{\gamma}\}$ and $\bar{b} = \{b_1, \ldots, b_{\gamma}\}$.\\
{\sc Output:} An element $z'$ separating tuples $\bar{a}$ and $\bar{b}$.\\
{\sc Initialization:} Set $S = \{(\bar{a},\bar{b},1)\}$. \\
{\sc Computations:}
\begin{enumerate}
    \item[A.]
    If $S = \emptyset$ then output FAILURE.
    \item[B.]
    Choose $(\bar{x},\bar{y},c) \in S$ such that $|\bar{x}| + |\bar{y}|$ is the minimal.
    \item[C.]
For each $i = 1,\ldots, n-1$ and $\varepsilon = \pm 1$
conjugate tuples $\bar{x}$ and $\bar{y}$ by a generator
$\sigma_i^\varepsilon$ and compute
    $$\delta_{i,\varepsilon} = |\bar{x}^{\sigma_i^\varepsilon}| + |\bar{y}^{\sigma_i^\varepsilon}| - ( |\bar{x}| + |\bar{y}| ).$$
    \item[D.]
If for some $\sigma_i^\varepsilon$ the sets
$\bar{x}^{\sigma_i^\varepsilon}$ and $\bar{y}^{\sigma_i^\varepsilon}$ are separated
then output $z' = \sigma_i^\varepsilon c$.
  \item[E.]
  Otherwise,  for each $i = 1,\ldots, n-1$ and $\varepsilon = \pm 1$
   add the tuple $(\bar{x}^{\sigma_i^\varepsilon},\bar{x}^{\sigma_i^\varepsilon},\sigma_i^\varepsilon c)$ to the set $S$.
 Goto step A.
\end{enumerate}
\end{algorithm}

We must mention here that, although there is a possibility that Algorithm \ref{alg:z_red_BACK}
outputs FAILURE or   does not terminate on some inputs, this situation has never occurred in our experiments.

Finally, we present another modification of Algorithm \ref{alg:z_red}.
\begin{algorithm}[Recovering conjugator - II]
\label{alg:z_red2}\ \\
{\sc Input:} Tuples $\bar{a} = \{a_1, \ldots, a_{\gamma}\}$ and $\bar{b} = \{b_1, \ldots, b_{\gamma}\}$.\\
{\sc Output:} An element $z'$ separating tuples $\bar{a}$ and $\bar{b}$.\\
{\sc Initialization:} Set $z' = 1$. \\
{\sc Computations:}
\begin{enumerate}
    \item[A.]
For each $i = 1,\ldots, n-1$ and $\varepsilon = \pm 1$
conjugate tuples $\bar{a}$ and $\bar{b}$ by a generator
$\sigma_i^\varepsilon$ and compute
    $$\delta_{i,\varepsilon} = |\bar{a}^{\sigma_i^\varepsilon}| + |\bar{b}^{\sigma_i^\varepsilon}| - ( |\bar{a}| + |\bar{b}| ).$$
    \item[B.]
If all $\delta_{i,\varepsilon}$ are positive (i.e., conjugation by $\sigma_i^\varepsilon$ cannot further decrease the total length)
and the sets $\bar{a}$ and $\bar{b}$ are separated
then output $z'$.
    \item[C.]
If all $\delta_{i,\varepsilon}$ are positive (i.e., conjugation by $\sigma_i^\varepsilon$ cannot further decrease the total length),
but the sets $\bar{a}$ and $\bar{b}$ are not separated then output FAILURE.
    \item[D.]
Otherwise choose $i$ and $\varepsilon$ for which $\delta_{i,\varepsilon}$ is minimal.
Set $z' = \sigma_i^\varepsilon z'$, $\bar{a} = \bar{a}^{\sigma_i^\varepsilon}$, and
$\bar{b} = \bar{b}^{\sigma_i^\varepsilon}$. Goto step A.
\end{enumerate}
\end{algorithm}

Algorithms \ref{alg:z_red} and \ref{alg:z_red2} are almost the same except that they
have different termination conditions. Algorithm  \ref{alg:z_red} stops as soon as the tuples are separated,
while Algorithm  \ref{alg:z_red2} tries to minimize the total length of the tuple
and when the minimal value is reached it checks if the current tuples are separated.

The complexity of step A in Algorithms \ref{alg:z_red} and \ref{alg:z_red2} is $O(\gamma n (|\overline{a}_i|+|\overline{b}_i|))$.
The maximal number of iterations can be bounded by the total length
of the input $|\overline{a}_i|+|\overline{b}_i|$. A very crude upper bound on the complexity of the two algorithms is
$O(\gamma n (|\overline{a}_i|+|\overline{b}_i|)^2)$.

The complexity of Algorithm \ref{alg:z_red_BACK} is harder to estimate. Potentially, the backtracking
mechanism may cause the algorithm to explore exponentially many potential solutions. However, our experiments show that a very few
backtracking steps are required to find a solution.

\subsection{Results of experiments}

The attack was tested on different sets of instances of the protocol. In particular
we generated the sets $BL$ and $BR$ randomly and used fixed sets
    $BL = \{\sigma_1, \ldots, \sigma_{l}\}$ and $BR =\{\sigma_{l+2}, \ldots, \sigma_{n-1}\}$.
We used the proposed values of the parameters (see Section \ref{sec:param_values}).
In addition the attack was tested on instances generated with the increased length of the secret conjugator $z$.

In all the experiments Algorithm \ref{alg:z_red_BACK} had $100\%$ success of
producing a separating conjugator $z'$. The average time of a run of the algorithm was 4.5 seconds when
executed on a Dual Core Opteron 2.2 GHz machine with 4GB of ram.
The algorithm without backtracking
had slightly smaller  but still respectable success rate of 90\%.
It is very interesting to notice that Algorithm \ref{alg:z_red2} actually recovered
the original secret conjugator $z$ in about $40\%$ of the cases.
That is the reason we mention this algorithm in the paper.

Experiments with instances of TTP protocol generated using $|z| = 50$ (which is almost three times
greater than the suggested value) again showed 100\% success rate.
However, we need to point out that the attack may fail when the length of $z$ is large
relative to the length of $\Delta^2$. For instance when in the second parameter set
the length of $z$ is increased to $100$,
the algorithm recovering $\Delta$-powers sometimes output wrong values.
Nevertheless, the success rate of Algorithm \ref{alg:z_red_BACK} is still about 90\% in this case.
We think it is possible to modify our algorithms to work with increased parameter values.
But the biggest concern here is that the protocol with increased parameter values
might be not suitable for purposes of lightweight cryptography.

\end{document}